\documentclass[11pt, a4paper]{amsart}
\usepackage{fullpage}
\usepackage{amssymb}
\usepackage{commath}
\usepackage{empheq}
\usepackage{amsmath, amsthm}
 \usepackage[foot]{amsaddr}
\usepackage{dsfont}
\usepackage{multicol}

\usepackage{etoolbox}
\makeatletter
\let\ams@starttoc\@starttoc
\makeatother
\usepackage[parfill]{parskip}
\makeatletter
\let\@starttoc\ams@starttoc
\patchcmd{\@starttoc}{\makeatletter}{\makeatletter\parskip\z@}{}{}
\makeatother

\usepackage{tikz}
\usetikzlibrary{matrix,arrows,decorations.pathmorphing,decorations.markings}
\usepackage{tikz-cd}

\usepackage{capt-of}

\usepackage{enumitem}

\usepackage{mathrsfs}

\usepackage{url}

\newcommand{\ul}{\underline}
\newcommand{\fas}{\mathcal{F}(as)}
\newcommand{\ifas}{\mathcal{IF}(as)}
\newcommand{\llb}{\left\lbrace}
\newcommand{\rrb}{\right\rbrace}

\renewcommand{\phi}{\varphi}

\newcommand{\C}{\mathbf{C}}
\newcommand{\RB}{\mathbb{R}\mathbb{B}}
\newcommand{\RNum}[1]{\MakeUppercase{\romannumeral #1}}

\newtheorem{thm}{Theorem}[section]

\theoremstyle{plain}
\newtheorem{prop}[thm]{Proposition}

\newtheorem{cor}[thm]{Corollary}

\theoremstyle{definition}
\newtheorem{defn}[thm]{Definition}
\theoremstyle{remark}
\newtheorem{rem}[thm]{Remark}

\newtheorem{eg}[thm]{Example}

\title{Categorifying equivariant monoids}
\author{Daniel Graves}
\address{School of Mathematics, University of Leeds, Woodhouse, Leeds, LS2 9JT, UK}
\email{dan.graves92@gmail.com}
\date{}

\begin{document}


\maketitle

\begin{abstract}
Equivariant monoids are very important objects in many branches of mathematics: they combine the notion of multiplication and the concept of a group action. In this paper we will construct categories which encode the structure borne by monoids with a group action by combining the theory of PROPs and PROBs with the theory of crossed simplicial groups. PROPs and PROBs are categories used to encode structures borne by objects in symmetric and braided monoidal categories respectively, whilst crossed simplicial groups are categories which encode a unital, associative multiplication and a compatible group action. We will produce PROPs and PROBs whose categories of algebras are equivalent to the categories of monoids, comonoids and bimonoids with group action using extensions of the symmetric and braid crossed simplicial groups. We will extend this theory to balanced braided monoidal categories using the ribbon braid crossed simplicial group. Finally, we will use the hyperoctahedral crossed simplicial group to encode the structure of an involutive monoid with a compatible group action. 
\end{abstract}

\section*{Introduction}
Equivariant monoids (also known as monoids with an action of a group $G$, or $G$-monoids) are very important objects in many branches of mathematics: they combine the notion of multiplication and the concept of a group action. Important examples in homological algebra and algebraic topology include equivariant algebras over a commutative ring and equivariant orthogonal ring spectra. These examples will be discussed further in Section \ref{alg-sec}.

In this paper we construct categories which encode the structure inherent in a $G$-monoid in any symmetric or braided monoidal category, by combining the theories of PROPs, PROBs and crossed simplicial groups.

PROPs are ``product and permutation" categories. They were originally introduced by Mac Lane in collaboration with Adams \cite[24]{MLCA}, to encode operations with multiple inputs and multiple outputs. PROPs come in a variety of flavours: the topological version was defined by Boardman and Vogt \cite{BV}; coloured PROPs have been studied by Johnson and Yau \cite{JY}; simplicial PROPs have been studied by Hackney and Robertson \cite{HR}; wheeled PROPs have studied by Markl, Merkulov and Shadrin \cite{MMS}. The notion of a PROP is more general than the notion of an operad, as we shall soon discuss. There is a rich theory of using PROPs to encode and link structures that are of interest in topology and algebra. For example, the papers \cite{Abrams} and \cite{RSW} provide PROPs for categories of Frobenius algebras (with extra structure) in terms of two-dimensional cobordisms and cospans of finite sets. These have proved to have applications further afield, such as in network theory \cite{BCR}. Another example comes from the papers \cite{MM1} and \cite{MM2} where PROPs are used to study $E_{\infty}$-structures. 

Crossed simplicial groups were introduced independently by Fiedorowicz and Loday \cite{FL} and Krasauskas \cite{Kras-skew}. These papers explore the relationship between crossed simplicial groups and equivariant homology. Crossed simplicial groups have also been used to study marked surfaces equipped with group actions \cite{DK}. A crossed simplicial group is a family of groups together with face and degeneracy maps that satisfy ``crossed" simplicial identities. Any crossed simplical group has an associated category that encodes a unital, associative multiplication compatible with group actions given by the family of groups. The foundational example is Connes' cyclic category \cite{Con}. In this paper, we are going to consider the families of symmetric groups, hyperoctahedral groups, braid groups and ribbon braid groups, all of which have interesting applications in algebraic topology. The symmetric and hyperoctahedral crossed simplicial groups are used in stable homotopy theory and equivariant stable homotopy theory; see \cite{Ault}, \cite{Ault-HO}, \cite{DMG-hyp} and \cite{DMG-e-inf}. The homology theory associated to the braid crossed simplicial group was originally used by Fiedorowicz to study two-fold loop-suspension spaces \cite{Fie}. The braid crossed simplicial group also arises in the study of configuration spaces \cite{BCWW}. The ribbon braid crossed simplicial group has been considered by Krasauskas in connection with framed braids and framed mapping class groups of smooth manifolds \cite{Krasauskas}.

The idea of combining the theory of PROPs and the theory of crossed simplicial groups to categorify categories of monoids goes back to Pirashvili \cite{Pir-PROP}. Pirashvili uses the structure of the symmetric crossed simplicial groups to construct PROPs whose categories of algebras are monoids, comonoids and bimonoids in a symmetric monoidal category. These results provided a generalization of previous results of Markl \cite{markl}, who constructed such PROPs in terms of generators and relations in order to study the deformation theory of algebras. Pirashvili's result shows that PROPs offer a strict generalization of operads. There is no operad whose algebras are bialgebras, but Pirashvili shows that there is such a PROP. Lack \cite{Lack} showed that Pirashvili's result fits into a broader categorical framework; that of composing PROPs via a distributive law. Using the techniques of Pirashvili and Lack, the author extended the existing theory to construct PROPs for monoids, comonoids and bimonoids equipped with an order-reversing involution, using the hyperoctahedral crossed simplicial group \cite{DMG-IFAS}. The author has also shown that one can extend the results of Lack's paper to the setting of braided monoidal categories and prove analogues of Pirashvili's result by using the braid crossed simplicial group \cite{DMG-PROB}. In this setting the categories we consider are called PROBs; ``product and braid" categories.

In this paper we will study equivariant extensions of the symmetric, braid and hyperoctahedral crossed simplicial groups and, in doing so, we completely generalize the existing results. In particular, our Theorems \ref{bimon-thm}, \ref{symm-bimon-thm} and \ref{hyp-bimon-thm} specialize to the theorems \cite[7.2]{DMG-PROB}, \cite[5.2]{Pir-PROP} and \cite[5.3]{DMG-IFAS} respectively, by taking the specific case of the trivial group. We also recall what it means for a braided monoidal category to be balanced and extend the results to this setting by considering the ribbon braid crossed simplicial group and its equivariant extensions.

The paper is organized as follows. In Section \ref{alg-sec} we recall the definitions of $G$-monoids, $G$-comonoids and $G$-bimonoids in a symmetric or braided monoidal category, together with examples of interest from topology and homological algebra. In Section \ref{PROP-sec} we recall the definitions of PROPs and PROBs. We recall the definition of an algebra for these categories in a symmetric or braided monoidal category. We provide a list of examples that will be used throughout the paper. In Section \ref{dist-law-sec} we describe the procedure for forming new PROPs and PROBs via a distributive law in two special cases. We show that the group composition in certain families of semi-direct products induces a distributive law between the associated groupoids. We also recall how the composition law in a crossed simplicial group determines a distributive law. We use these distributive laws to form new PROPs and PROBs from the examples in Section \ref{PROP-sec}. In Section \ref{mon-comon-sec} we prove that we have constructed PROPs and PROBs whose categories of algebras are equivalent to the categories of $G$-monoids and $G$-comonoids. We prove this for symmetric monoidal categories, braided monoidal categories and balanced braided monoidal categories. We also construct a PROP for  $G$-commutative monoids and $G$-cocommutative comonoids in a symmetric monoidal category. In Section \ref{braid-bimon-sec} we construct the PROBs for $G$-bimonoids in braided and balanced braided monoidal categories by composing the PROBs for $G$-monoids and  $G$-comonoids. In Section \ref{symm-bimon-sec} we prove similar results for $G$bimonoids in a symmetric monoidal category. We construct the necessary PROPs in two different ways. In the first instance we provide a proof that is similar to the one given for braided monoidal categories; by composing the PROPs for $G$-monoids and $G$-comonoids. However, in the symmetric monoidal case we have an interesting alternative description of these PROPs in terms of categories of non-commutative sets and double categories. Using this construction also allows us to give a nice description of the PROPs for $G$-bimonoids that are commutative, cocommutative or both. In Section \ref{hyp-bimon-sec} we note that in a symmetric monoidal category we have monoids that possess both an involution and a group action in a compatible way. Group algebras provide one important example. We use the equivariant extensions of the hyperoctahedral crossed simplicial group to construct a PROP which encodes this data.

\subsection*{Acknowledgements}
I would like to thank James Cranch and Sarah Whitehouse for interesting and helpful conversations whilst writing this paper.

\section{Equivariant monoids}
\label{alg-sec}
In this section we introduce some notation for categories of $G$-monoids, $G$-comonoids and $G$-bimonoids. We also give some examples from homological algebra and algebraic topology.

\begin{defn}
Let $\mathbf{C}$ be a braided or symmetric monoidal category. We denote by $\mathbf{Mon}\left(\mathbf{C}\right)$, $\mathbf{Comon}\left(\mathbf{C}\right)$ and $\mathbf{Bimon}\left(\mathbf{C}\right)$ the categories of monoids, comonoids and bimonoids in $\mathbf{C}$ respectively.
\end{defn}

\begin{defn}
Let $\mathbf{C}$ be a braided or symmetric monoidal category and let $G$ be a group. We denote the functor category $\mathrm{Fun}\left(G, \mathbf{Mon}(\mathbf{C})\right)$ by $\mathbf{GMon}$ and call it the \emph{category of $G$-monoids in $\mathbf{C}$}.
Similarly we denote the functor categories $\mathrm{Fun}\left(G, \mathbf{Comon}(\mathbf{C})\right)$ and $\mathrm{Fun}\left(G, \mathbf{Bimon}(\mathbf{C})\right)$ by $\mathbf{GComon}\left(\mathbf{C}\right)$ and $\mathbf{GBimon}\left(\mathbf{C}\right)$ respectively and call them the \emph{categories of $G$-comonoids and $G$-bimonoids in $\C$}.
\end{defn}

\begin{eg}
We note that we can think of monoids in a symmetric or braided monoidal category as $G$-monoids, with $G$ being the trivial group.
\end{eg}

\begin{eg}
The following three examples demonstrate the importance of $G$-monoids in algebraic topology and homological algebra.
\begin{enumerate}
\item Let $k$ be a commutative ring and let $\C$ be the symmetric monoidal category of $k$-modules, $\mathbf{Mod}_k$. In this case, a $G$-monoid is a $G$-algebra. Explicitly, a $G$-algebra is an associative $k$-algebra $A$ equipped with a group homomorphism $G\rightarrow \mathrm{Aut}(A)$. These have been studied by Inassaridze in connection with equivariant homological algebra \cite{Inassaridze}.
\item Let $G$ be a group. Consider the braided monoidal category of crossed $G$-sets. This is described in \cite[4.2]{FY} and used to define invariants in low dimensional topology. A monoid in this category is a crossed $G$-monoid. These are described in \cite[Definition 4.1]{Bouc}. One can consider these as $G$-monoids in the category of crossed $G$-sets via the multiplication in $G$. Bouc uses crossed $G$-monoids to construct certain Green functors and Mackey functors and uses these to study the Hochschild cohomology of $G$ over a ring.
\item Consider the symmetric monoidal category of orthogonal spectra, $\mathbf{Sp}^{O}$. This category first appeared in \cite{MMSS}. A monoid in $\mathbf{Sp}^{O}$ is an orthogonal ring spectrum. A $G$-monoid in $\mathbf{Sp}^{O}$ is a $G$-equivariant orthogonal ring spectrum. Equivariant orthogonal spectra were first studied in \cite{MM-equiv}. These are important objects in equivariant stable homotopy theory. For example, they are used in the construction of the $\infty$-category of genuine $G$-equivariant spectra; see \cite[Section \RNum{2}.2]{NS} or \cite[Part 2]{MNN} for example.
\end{enumerate}
\end{eg}

\section{PROPs and PROBs}
\label{PROP-sec}
In this section we recall the definitions of PROPs and PROBs. We recall the definition of an algebra for such a category in a symmetric or braided monoidal category. We give a list of the examples that will be used throughout the paper. 

\begin{defn}
\label{sets-defn}
For $n\geqslant 1$ we define $\ul{n}$ to be the set $\llb 1,\dotsc ,n\rrb$. We define $\ul{0}=\emptyset$.
\end{defn}

\begin{defn}
Let $\mathbf{M}$ be a monoidal category, let $\mathbf{S}$ be a symmetric monoidal category and let $\mathbf{B}$ be a braided monoidal category.
\begin{enumerate}
\item A PRO $\mathbf{T}$ is a strict monoidal category whose objects are the sets $\ul{n}$ for $n\geqslant 0$ and whose tensor product is given by addition. An \emph{algebra of $\mathbf{T}$ in $\mathbf{M}$} is a strict monoidal functor $\mathbf{T}\rightarrow \mathbf{M}$.
\item A \emph{PROP} $\mathbf{P}$ is a symmetric strict monoidal category whose objects are the sets $\ul{n}$ for $n\geqslant 0$ with tensor product given by addition. A \emph{$\mathbf{P}$-algebra in $\mathbf{S}$} is a symmetric strict monoidal functor $\mathbf{P}\rightarrow\mathbf{S}$. We denote the category of $\mathbf{P}$-algebras in $\mathbf{S}$ and natural transformations by $\mathbf{Alg}\left(\mathbf{P},\mathbf{S}\right)$.
\item A \emph{PROB} $\mathbf{Q}$ is a braided strict monoidal category whose objects are the sets $\ul{n}$ for $n\geqslant 0$ with tensor product given by addition. A \emph{$\mathbf{Q}$-algebra in $\mathbf{B}$} is a braided strict monoidal functor $\mathbf{Q}\rightarrow\mathbf{B}$. We denote the category of $\mathbf{Q}$-algebras in $\mathbf{B}$ and natural transformations by $\mathbf{Alg}\left(\mathbf{Q},\mathbf{B}\right)$.
\end{enumerate}
\end{defn}

\begin{eg}
\label{PROs-eg}
The following examples are the building blocks for the results of this paper.
\begin{enumerate}
\item Following \cite[2.2]{Lack} we denote by $\mathbb{D}$ the PRO of finite ordinals and order-preserving maps. For a strict monoidal category $\mathbf{M}$, an algebra of $\mathbb{D}$ in $\mathbf{M}$ is a monoid in $\mathbf{M}$, see \cite[VII 5]{ML}.
\item Following \cite[5.1]{Lack}, let $\mathbb{F}$ denote the skeletal category of finite sets. For a symmetric monoidal category $\mathbf{S}$, an algebra of $\mathbb{F}$ in $\mathbf{S}$ is a commutative monoid.
\item Let $G$ be a group. Let $\mathbb{G}$ be the PRO such that $\mathrm{Hom}_{\mathbb{G}}\left(\ul{n} , \ul{m}\right)$ is empty if $n\neq m$ and $\mathrm{Hom}_{\mathbb{G}}\left(\ul{n} , \ul{n}\right) = G^n$. The disjoint union of morphisms corresponds to the product of group elements.
\item Following \cite[2.4]{Lack}, let $\mathbb{P}$ denote the PROP of finite sets and bijections. Observe that this is the subobject of $\mathbb{F}$ consisting of the isomorphisms in the skeletal category of finite sets.
\item Following \cite[2.11]{DMG-IFAS}, let $\mathbb{H}$ denote the PROP of hyperoctahedral groups. Recall that the hyperoctahedral group $H_n$ is the semi-direct product $C_2^n\rtimes \Sigma_n$.
\item Following \cite[Example 2.1]{JS}, let $\mathbb{B}$ denote the PROB of braid groups. The set $\mathrm{Hom}_{\mathbb{B}}\left(\ul{n},\ul{m}\right)$ is empty for $n\neq m$ and $\mathrm{Hom}_{\mathbb{B}}\left(\ul{n} , \ul{n}\right)=B_n$, the braid group on $n$ strings. The strict monoidal structure is given by the \emph{addition of braids}. 
\item Following \cite[Example 6.4]{JS}, let $\RB$ denote the PROB of ribbon braid groups. The set $\mathrm{Hom}_{\RB}\left(\ul{n},\ul{m}\right)$ is empty for $n\neq m$ and $\mathrm{Hom}_{\mathbb{\RB}}\left(\ul{n} , \ul{n}\right)=RB_n$, the ribbon braid group on $n$ ribbons. The strict monoidal structure is given by addition of ribbons. The ribbon braid group $RB_n$ can be identified with the semi-direct product $\mathbb{Z}^n\rtimes B_n$ where the integer labels indicate the number of full twists to perform on each ribbon.
\end{enumerate}
\end{eg}

\section{Distributive laws and crossed simplicial groups}
\label{dist-law-sec}
We can form new PROs, PROPs and PROBs via the notion of a \emph{distributive law} as defined in \cite[Section 3]{Lack} and \cite[Section 2]{RW}. In particular we will form composites from the the categories defined in Example \ref{PROs-eg}. 

\subsection{Semi-direct products}
Let $G$ be a group and consider the semi-direct products $G^n\rtimes \Sigma_n$, $G^n\rtimes B_n$, $G^n\rtimes RB_n$ and $G^n\rtimes H_n$ where in each case the tuple $G^n$ is acted on by the underlying permutation. These groups are sometimes also called wreath products. We can think of these as labelled permutations (\cite[Theorem 3.10]{FL}), labelled braids (\cite[Example 2.1]{JS}), labelled ribbons (\cite[Example 6.4]{JS}) and labelled hyperoctahedral group elements. 

The definition of composition in these semi-direct products satisfies the data of a distributive law. For example, if we consider the symmetric case, a pair $\left(\mathbf{x} , \sigma\right)$ where $\sigma\in \mathrm{Hom}_{\mathbb{P}}\left(\ul{n} ,\ul{n}\right)$ and $\mathbf{x}\in \mathrm{Hom}_{\mathbb{G}}\left(\ul{n} , \ul{n}\right)$ determines a unique pair $\left(\sigma , \sigma\left(\mathbf{x}\right)\right)$ where $\sigma$ remains unchanged and $\sigma$ acts on the element $\mathbf{x}\in G^n$ by permutation. One can readily check that this assignment is compatible with the monoidal structure of both categories. Following \cite[3.8]{Lack}, we have a well-defined PRO $\mathbb{P}\otimes \mathbb{G}$, with $\mathrm{Hom}_{\mathbb{P}\otimes \mathbb{G}}\left(\ul{n} , \ul{m}\right) = \emptyset$ if $n\neq m$ and $\mathrm{Hom}_{\mathbb{P}\otimes \mathbb{G}}\left(\ul{n} , \ul{n}\right) = G^n\rtimes \Sigma_n$. In fact, this category has a canonical PROP structure induced from $\mathbb{P}$

Mimicking this procedure we form a PROP $\mathbb{H}\otimes \mathbb{G}$ and PROBs $\mathbb{B}\otimes \mathbb{G}$ and $\RB\otimes \mathbb{G}$.

\subsection{Crossed simplicial groups}
Consider a family of groups $\llb J_n\rrb$ indexed over the natural numbers. Let $\mathbb{J}$ be the groupoid whose objects are the sets $\ul{n}$ with $\mathrm{Hom}_{\mathbb{J}}\left(\ul{n} , \ul{n}\right) = J_n$. We use $J$ for a group here because we are already using $G$ for a group, and $H$ for hyperoctahedral groups. One of the equivalent definitions for $\llb J_n \rrb$ to be a crossed simplicial group is the existence of a distributive law $\mathbb{J}\otimes \mathbb{D}\rightarrow \mathbb{D}\otimes \mathbb{J}$. In other words, given a pair of morphisms $\left(j, \phi\right)$ with $\phi \in \mathrm{Hom}_{\mathbb{D}}\left(\ul{n} , \ul{m}\right)$ and $j\in \mathrm{Hom}_{\mathbb{J}}\left(\ul{m} , \ul{m}\right)$ we have a unique pair $\left(j_{\star}(\phi),\phi^{\star}(j)\right)$ with $\phi^{\star}(j) \in \mathrm{Hom}_{\mathbb{J}}\left(\ul{n}, \ul{n}\right)$ and $j_{\star}(\phi) \in \mathrm{Hom}_{\mathbb{D}}\left(\ul{n}, \ul{m}\right)$, satisfying the relations of \cite[2.4]{RW} and compatible with the monoidal structure of both categories. These assignments turn $\mathbb{D}\otimes \mathbb{J}$ into a well-defined category. The objects are the sets $\ul{n}$. An element of $\mathrm{Hom}_{\mathbb{D}\otimes \mathbb{J}}\left(\ul{n} , \ul{m}\right)$ is a pair $\left(\phi , j\right)$ with $j\in \mathrm{Hom}_{\mathbb{J}}\left(\ul{n} , \ul{n}\right)$ and $\phi\in \mathrm{Hom}_{\mathbb{D}}\left(\ul{n} , \ul{m}\right)$, with composition determined by the distributive law. 

The descriptions of these maps for the families of symmetric groups, hyperoctahedral groups and braid groups were given in \cite[3.1,\,3.3,\,3.7]{FL} respectively. For the case of ribbon braids one follows \cite[3.7]{FL}, replacing the identity braid with the identity ribbon.

By a theorem of Fiedorowicz and Loday \cite[Theorem 3.10]{FL}, we know that the family of groups $\llb G^n\rtimes \Sigma_n\rrb$ form a crossed simplicial group. In particular, the proof describes assignments $\phi^{\star}$ and $g_{\star}$ which describe a distributive law
\[\left(\mathbb{P}\otimes \mathbb{G}\right) \otimes \mathbb{D} \rightarrow \mathbb{D}\otimes \left(\mathbb{P}\otimes \mathbb{G}\right),\]
compatible with the monoidal structure. These assignments are constructed from those in the symmetric crossed simplicial group by extending them to also take into account the $G$-labels. We  therefore have a well-defined PROP $\mathbb{D}\otimes \mathbb{P}\otimes \mathbb{G}$, whose set of morphisms $\mathrm{Hom}_{\mathbb{D}\otimes \mathbb{P}\otimes \mathbb{G}}\left(\ul{n} , \ul{m}\right)$ consists of pairs $(\phi, g)$ with $g\in \mathrm{Hom}_{\mathbb{P}\otimes \mathbb{G}}\left(\ul{n} , \ul{n}\right)= G^n\rtimes \Sigma_n$ and $\phi \in \mathrm{Hom}_{\mathbb{D}}\left(\ul{n}, \ul{m}\right)$.

We observe that the proof \cite[Theorem 3.10]{FL} works equally well when we replace $\mathbb{P}\otimes \mathbb{G}$ with the categories $\mathbb{H}\otimes \mathbb{G}$, $\mathbb{B}\otimes \mathbb{G}$ and $\RB\otimes \mathbb{G}$. In this way we form a PROP $\mathbb{D}\otimes \mathbb{H} \otimes \mathbb{G}$ and PROBs $\mathbb{D}\otimes  \mathbb{B}\otimes \mathbb{G}$ and $\mathbb{D}\otimes \RB \otimes \mathbb{G}$. 

\begin{rem}
\label{decomp-rem}
Note that any morphism $\left(\phi , j\right)$ can be decomposed in terms of precisely three morphisms. We have the group element $j$ as an automorphism. An order-preserving map $\phi$ has a unique decomposition in terms of face and degeneracy maps \cite[B.2]{Lod}. However, using the monoidal structure in the category $\mathbb{D}$, any face map can be written as a tensor product of the unique morphism $m \in \mathrm{Hom}_{\mathbb{D}}\left(\ul{2} , \ul{1}\right)$ with identity morphisms and any  degeneracy maps can be expressed as a tensor product of the unique map $u\in \mathrm{Hom}_{\mathbb{D}}\left(\ul{0}, \ul{1}\right)$ with identity morphisms. Therefore, using the unique decomposition of order-preserving maps and the monoidal structure we can write any morphism in a crossed simplicial group as a finite expression in terms of $m$, $u$ and the group elements.
\end{rem}

\subsection{A distributive law for skeletal finite sets}
We define a distributive law between the PROP given by the skeletal category of finite sets, $\mathbb{F}$, and the groupoid $\mathbb{G}$. Given a pair $\left(\mathbf{x}, f\right)$ with $f\in \mathrm{Hom}_{\mathbb{F}}\left(\ul{n} , \ul{m}\right)$ and 
\[\mathbf{x} = \left(g_1,\dotsc, g_m\right) \in G^m,\]
we have a unique pair $\left(f, \mathbf{x}^{\prime}\right)$ where $f$ has remained unchanged and 
\[\mathbf{x}^{\prime} = \left(g_{f(1)}, \dotsc , g_{f(n)}\right)\in G^n.\]
One can readily check that this satisfies the relations of \cite[2.4]{RW} and is compatible with the monoidal structure of both categories. We therefore have a well-defined PROP $\mathbb{F}\otimes \mathbb{G}$.

\section{Categorifying equivariant monoids and comonoids}
\label{mon-comon-sec}

We show that the PROP $\mathbb{D}\otimes  \mathbb{P}\otimes \mathbb{G}$ and the PROB $\mathbb{D}\otimes  \mathbb{B}\otimes \mathbb{G}$ encode the structures of $G$-monoids and $G$-comonoids in symmetric and braided monoidal categories. We show that if we take the PROP $\mathbb{F}\otimes \mathbb{G}$ we encode the structures of  $G$-commutative monoids and $G$-cocommutative comonoids. We also recall the notion of a balanced braided monoidal category and show that the PROBs $\mathbb{D}\otimes \RB$ and $\mathbb{D}\otimes \RB \otimes \mathbb{G}$ encode the structures of monoids and $G$-monoids in a balanced braided monoidal category.

\begin{prop}
\label{comp-prop}
Let $\mathbf{S}$ be a symmetric monoidal category, let $\mathbf{B}$ be a braided monoidal category and let $G$ be a group. There are equivalences of categories 
\begin{enumerate}
\item $\mathbf{Alg}\left( \mathbb{D}\otimes \mathbb{P}\otimes \mathbb{G} , \mathbf{S}\right) \simeq \mathbf{GMon}\left(\mathbf{S}\right)$;
\item $\mathbf{Alg}\left( \mathbb{D}\otimes \mathbb{B}\otimes \mathbb{G} , \mathbf{B}\right) \simeq \mathbf{GMon}\left(\mathbf{B}\right)$;
\item $\mathbf{Alg}\left( \left(\mathbb{D}\otimes \mathbb{P}\otimes \mathbb{G}\right)^{op} , \mathbf{S}\right) \simeq \mathbf{GComon}\left(\mathbf{S}\right)$ and
\item $\mathbf{Alg}\left( \left(\mathbb{D}\otimes \mathbb{B}\otimes \mathbb{G}\right)^{op} , \mathbf{B}\right) \simeq \mathbf{GComon}\left(\mathbf{B}\right)$.
\end{enumerate}
\end{prop}
\begin{proof}
The proof for monoids in both cases is similar to \cite[5.5]{Lack}, \cite[2.13]{DMG-IFAS} and \cite[6.2]{DMG-PROB}. We will give the details of the symmetric monoidal case and indicate the necessary changes for the braided monoidal case. Let $U\left(\mathbb{D}\otimes \mathbb{P}\otimes \mathbb{G}\right)$ denote the underlying PRO of the PROP $\mathbb{D}\otimes \mathbb{P}\otimes \mathbb{G}$. By \cite[3.10]{Lack}, a $U\left(\mathbb{D}\otimes \mathbb{P}\otimes \mathbb{G}\right)$-algebra structure on an object $M$ of $\mathbf{S}$ consists of a $\mathbb{D}$-algebra structure and a $\left(\mathbb{P}\otimes \mathbb{G}\right)$-algebra structure subject to a compatibility condition. A $\mathbb{D}$-algebra structure is a monoid structure. A $\left(\mathbb{P}\otimes \mathbb{G}\right)$-algebra is an object $M$ together with a morphism $g\colon M\rightarrow M$ for each element $g\in G$ and, for each element $\left(\sigma, \left(g_1,\dotsc , g_n\right)\right)$, a morphism $M^{\otimes n}\rightarrow M^{\otimes n}$ given by applying $(g_1,\dotsc ,g_n)$ component-wise followed by an isomorphism determined by $\sigma \in\Sigma_n$. By \cite[5.5]{Lack} a $U\left(\mathbb{D}\otimes \mathbb{P}\otimes \mathbb{G}\right)$-algebra structure is a $\left(\mathbb{D}\otimes \mathbb{P}\otimes \mathbb{G}\right)$-algebra structure if and only if the isomorphisms induced from the elements $\sigma \in \Sigma_n$ are the symmetry isomorphisms. The compatibility condition ensures that the morphisms $g\colon M\rightarrow M$ form a group action compatible with the symmetric monoidal structure. Finally, a morphism in $\mathbf{S}$ is a map of $G$-monoids if and only if it respects the $\mathbb{D}$-algebra structure and the $\left(\mathbb{P}\otimes \mathbb{G}\right)$-algebra structure. By \cite[3.12]{Lack}, this is true if and only if it respects the $\left(\mathbb{D}\otimes \mathbb{P}\otimes \mathbb{G}\right)$-algebra structure. The braided monoidal case is similar, using \cite[5.10]{DMG-PROB} in place of \cite[3.12]{Lack}.

For comonoids the proof is also similar for both the symmetric and braided monoidal cases. We give the proof for the symmetric monoidal case. For categories $\mathbf{C}$ and $\mathbf{D}$ there is an isomorphism of functor categories $\mathrm{Fun}\left(\C^{op}, \mathbf{D}^{op}\right) \cong \mathrm{Fun}\left(\C , \mathbf{D}\right)^{op}$ and so
\[\mathrm{Fun}\left(G , \mathbf{Comon}(\mathbf{S})\right) =\mathrm{Fun}\left(G , \mathbf{Mon}(\mathbf{S}^{op})^{op}\right)\cong\mathrm{Fun}\left(G^{op} , \mathbf{Mon}(\mathbf{S}^{op})\right)^{op}.\]
Since a group is isomorphic to its opposite we see that 
\[\mathrm{Fun}\left(G^{op} , \mathbf{Mon}(\mathbf{S}^{op})\right)^{op} \cong\mathrm{Fun}\left(G , \mathbf{Mon}(\mathbf{S}^{op})\right)^{op}=\mathbf{GMon}\left(\mathbf{S}^{op}\right)^{op}.\]
Furthermore, by Proposition \ref{comp-prop},
\[\mathbf{GMon}\left(\mathbf{S}^{op}\right)^{op} \simeq \mathbf{Alg}\left(\mathbb{D}\otimes \mathbb{P}\otimes \mathbb{G} , \mathbf{S}^{op}\right)^{op} = \mathbf{Alg}\left(\left(\mathbb{D}\otimes \mathbb{P}\otimes \mathbb{G}\right)^{op} , \mathbf{S}\right)\]
as required.
\end{proof}

\begin{prop}
\label{commutative-prop}
Let $\mathbf{S}$ be a symmetric monoidal category and let $G$ be a group.
\begin{enumerate}
\item The category $\mathbf{Alg}\left(\mathbb{F}\otimes \mathbb{G}, \mathbf{S}\right)$ is equivalent to the category of $G$-commutative monoids in $\mathbf{S}$;
\item the category $\mathbf{Alg}\left(\left(\mathbb{F}\otimes \mathbb{G}\right)^{op}, \mathbf{S}\right)$ is equivalent to the category of $G$-cocommutative comonoids in $\mathbf{S}$.
\end{enumerate}
\end{prop}
\begin{proof}
The proof follows the same method as Proposition \ref{comp-prop}, noting that an algebra for the PROP $\mathbb{F}$ is a commutative monoid, rather than a monoid.
\end{proof}

Recall from \cite[Definition 6.1]{JS} that a braided monoidal category $\mathbf{B}$ is called \emph{balanced} if for each object $A$ we have a \emph{full twist isomorphism} $\Theta_A\colon A\rightarrow A$, such that the full twist on the unit object is the identity and such that the full twists are compatible with the braiding. The category of braids $\mathbb{B}$, the category of ribbon braids $\RB$ and the category of labelled ribbon braids $\RB\otimes \mathbb{G}$ are all examples of balanced braided monoidal categories. Furthermore, any symmetric monoidal category is balanced. We can use the PROB of ribbon braids to identify the extra structure present in balanced braided monoidal categories. 

\begin{prop}
\label{balanced-prop}
Let $\mathbf{B}$ be a balanced braided monoidal category and let $G$ be a group. There are equivalences of categories 
\begin{enumerate}
\item $\mathbf{Alg}\left( \mathbb{D}\otimes \RB , \mathbf{B}\right) \simeq \mathbf{Mon}\left(\mathbf{B}\right)$;
\item $\mathbf{Alg}\left( \mathbb{D}\otimes \RB\otimes \mathbb{G} , \mathbf{B}\right) \simeq \mathbf{GMon}\left(\mathbf{B}\right)$;
\item $\mathbf{Alg}\left( \left(\mathbb{D}\otimes \RB\right)^{op} , \mathbf{B}\right) \simeq \mathbf{Comon}\left(\mathbf{B}\right)$ and
\item $\mathbf{Alg}\left( \left(\mathbb{D}\otimes \RB\otimes \mathbb{G}\right)^{op} , \mathbf{B}\right) \simeq \mathbf{GComon}\left(\mathbf{B}\right)$.
\end{enumerate}
\end{prop}
\begin{proof}
The non-equivariant cases are similar to \cite[6.2]{DMG-PROB} and \cite[6.3]{DMG-PROB}. The added data of twists in the ribbons corresponds precisely to applying the twist isomorphisms in $\mathbf{B}$. The equivariant cases are similar to Proposition \ref{comp-prop}.
\end{proof}

\section{Equivariant bimonoids in a braided monoidal category}
\label{braid-bimon-sec}
In this section we construct a PROB for $G$-bimonoids in a braided monoidal category. We do this by defining a distributive law between $\mathbb{D}\otimes \mathbb{B}\otimes \mathbb{G}$ and its opposite, that is between the PROBs for $G$-monoids and $G$-comonoids, and analysing the composite PROB using the machinery introduced in \cite[Section 5]{DMG-PROB}. 

\begin{defn}
For ease of notation, let $Q$ denote the category $\mathbb{D}\otimes \mathbb{B}\otimes \mathbb{G}$. Let $Q^{op}\otimes_{\mathbb{B}\otimes \mathbb{G}} Q$ denote the category of equivalence classes of cospans in $Q$. Let $Q\otimes_{\mathbb{B}\otimes \mathbb{G}} Q^{op}$ denote the category of equivalence classes of spans in $Q$.
\end{defn}

Following Remark \ref{decomp-rem} it suffices to define the necessary distributive law in terms of the morphisms $m\in \mathrm{Hom}_{\mathbb{D}}\left(\ul{2} , \ul{1}\right)$, $u\in \mathrm{Hom}_{\mathbb{D}}\left(\ul{0} , \ul{1}\right)$ and the morphisms of $\mathbb{B}\otimes \mathbb{G}$. In the following definition we extend the distributive law of \cite[7.1]{DMG-PROB} to take into account the labelling data. Observe that this distributive law is completely determined by the crossed simplicial group structure of $\mathbb{D}\otimes \mathbb{B}\otimes \mathbb{G}$

\begin{defn}
\label{dist-law-defn}
We define a distributive law of PROBs
\[Q^{op}\otimes_{\mathbb{B}\otimes \mathbb{G}} Q\rightarrow Q\otimes_{\mathbb{B}\otimes \mathbb{G}} Q^{op}\]
to be determined as follows
\begin{enumerate}
\item \label{braid1} For $g$, $h\in G^n\rtimes B_n$
\begin{center}
\begin{tikzcd}
\ul{n}\arrow[rr, "g"]&&\ul{n}&&\ul{n}\arrow[ll, "h",swap] & \mapsto & \ul{n}&&\ul{n}\arrow[ll,"{g^{-1}}",swap]\arrow[rr, "{h^{-1}}"]&&\ul{n}
\end{tikzcd}
\end{center}
\item \label{mixed} For $\phi \in \mathrm{Hom}_{\mathbb{D}}\left(\ul{n},\ul{m}\right)$ and $g\in G^m\rtimes B_m$
\begin{center}
\begin{tikzcd}
\ul{n}\arrow[rr, "{\phi}"]&&\ul{m}&&\ul{m}\arrow[ll, "g",swap] & \mapsto  & \ul{n}&&\ul{n}\arrow[ll,"{g_{\star}(\phi)}",swap]\arrow[rr, "{\phi^{\star}(g)}"]&&\ul{m}\\
\ul{m}\arrow[rr, "{g}"]&&\ul{m}&&\ul{n}\arrow[ll, "{\phi}",swap] & \mapsto  & \ul{m}&&\ul{n}\arrow[ll,"{\phi^{\star}(g)}",swap]\arrow[rr, "{g_{\star}(\phi)}"]&&\ul{n}
\end{tikzcd}
\end{center}
where $g_{\star}(\phi)$ and $\phi^{\star}(g)$ are the assignments from the crossed simplicial group structure of $\mathbb{D}\otimes \mathbb{B}\otimes \mathbb{G}$.
\item \label{ord-pres1} For $m\in \mathrm{Hom}_{\mathbb{D}}\left(\ul{2},\ul{1}\right)$ and $u\in \mathrm{Hom}_{\mathbb{D}}\left(\ul{0},\ul{1}\right)$
\begin{center}
\begin{tikzcd}
\ul{2}\arrow[rr, "m"]&&\ul{1}&&\ul{2}\arrow[ll, "m",swap] & \mapsto  & \ul{2}&&\ul{4}\arrow[ll,"{m^{\amalg 2}\circ \sigma_{2,3}}",swap]\arrow[rr, "{m^{\amalg 2}}"]&&\ul{2}\\
\ul{0}\arrow[rr, "u"]&&\ul{1}&&\ul{0}\arrow[ll, "u",swap] & \mapsto & \ul{0}&&\ul{0}\arrow[ll,"{id_0}",swap]\arrow[rr, "{id_0}"]&&\ul{0}\\
\ul{2}\arrow[rr, "m"]&&\ul{1}&&\ul{0}\arrow[ll, "u",swap] & \mapsto & \ul{2}&&\ul{0}\arrow[ll,"{u^2}",swap]\arrow[rr, "{id_0}"]&&\ul{0}\\
\ul{0}\arrow[rr, "u"]&&\ul{1}&&\ul{2}\arrow[ll, "m",swap] & \mapsto & \ul{0}&&\ul{0}\arrow[ll,"{id_0}",swap]\arrow[rr, "{u^2}"]&&\ul{2}
\end{tikzcd}
\end{center}
where $\sigma_{2,3}$ is notation for $id_1\amalg \sigma \amalg id_1$ and $\sigma$ is the generator of the braid group $B_2$.
\end{enumerate}
\end{defn}

\begin{thm}
\label{bimon-thm}
Let $\mathbf{B}$ be a braided monoidal category. Let $G$ be a group. There is an equivalence of categories $\mathbf{Alg}\left(\mathbb{D} \otimes \left(\mathbb{B}\otimes \mathbb{G}\right)\otimes \mathbb{D}^{op} , \mathbf{B}\right) \simeq \mathbf{GBimon}\left(\mathbf{B}\right)$.
\end{thm}
\begin{proof}
The distributive law of Definition \ref{dist-law-defn} implies that $Q\otimes_{\mathbb{B}\otimes \mathbb{G}} Q^{op}$ is a composite PROB by \cite[5.9]{DMG-PROB}. By \cite[5.10]{DMG-PROB} a $\left(Q\otimes_{\mathbb{B}\otimes \mathbb{G}} Q^{op}\right)$-algebra structure for an object $M$ in $\mathbf{B}$ is a $Q$-algebra structure, that is a $G$-monoid structure, and a $Q^{op}$-algebra structure, that is a $G$-comonoid structure, subject to compatibility conditions arising from the distributive law of Definition \ref{dist-law-defn}. These compatibility conditions are precisely those requiring $M$ to be a $G$-bimonoid. Finally we observe that a morphism in $\mathbf{B}$ is a morphism of $G$-bimonoids if and only if it preserves the $Q$-algebra structure and the $Q^{op}$-algebra structure. By \cite[5.10]{DMG-PROB} this is true if and only if it preserves the $\left(Q\otimes_{\mathbb{B}\otimes \mathbb{G}} Q^{op}\right)$-algebra structure.

The isomorphisms in $Q$ and $Q^{op}$ are the labelled braid groups. Therefore, when we take equivalence classes of spans we are identifying the two group actions. We therefore have an isomorphism of PROBs 
\[Q\otimes_{\mathbb{B}\otimes \mathbb{G}} Q^{op} \cong \mathbb{D} \otimes \left(\mathbb{B}\otimes \mathbb{G}\right)\otimes \mathbb{D}^{op}\]
from which the result follows.
\end{proof}

\begin{thm}
Let $\mathbf{B}$ be a balanced braided monoidal category. There are equivalences of categories 
\begin{enumerate}
\item $\mathbf{Alg}\left(\mathbb{D} \otimes \RB\otimes \mathbb{D}^{op} , \mathbf{B}\right) \simeq \mathbf{Bimon}\left(\mathbf{B}\right)$.
\item $\mathbf{Alg}\left(\mathbb{D} \otimes \left(\RB\otimes \mathbb{G}\right)\otimes \mathbb{D}^{op} , \mathbf{B}\right) \simeq \mathbf{GBimon}\left(\mathbf{B}\right)$.
\end{enumerate}
\end{thm}
\begin{proof}
We define appropriate distributive laws by replacing the labelled braids in Definition \ref{dist-law-defn} by the ribbon braids and the labelled ribbon braids respectively, and using the composition in the categories $\mathbb{D}\otimes \RB$ and $\mathbb{D}\otimes \RB \otimes \mathbb{G}$. The proof then follows similarly to Theorem \ref{bimon-thm}, making use of Proposition \ref{balanced-prop}.
\end{proof}

\section{Equivariant bimonoids in a symmetric monoidal category}
\label{symm-bimon-sec}
In this section we construct a PROP for $G$-bimonoids in a symmetric monoidal category. We begin by stating this result directly, following Theorem \ref{bimon-thm}. However, in the symmetric monoidal case we have an interesting alternative description of this PROP in terms of non-commutative sets after the fashion of \cite{Pir-PROP} and \cite{DMG-IFAS}. We also include results for $G$-commutative and $G$-cocommutative bimonoids.

\subsection{A composite PROP for equivariant bimonoids}
We begin by stating the symmetric monoidal analogue of Theorem \ref{bimon-thm}.

\begin{thm}
\label{comm-bimon-thm}
Let $\mathbf{S}$ be a symmetric monoidal category. Let $G$ be a group. There is an equivalence of categories $\mathbf{Alg}\left(\mathbb{D} \otimes \left(\mathbb{P}\otimes \mathbb{G}\right)\otimes \mathbb{D}^{op} , \mathbf{S}\right) \simeq \mathbf{GBimon}\left(\mathbf{S}\right)$.
\end{thm}
\begin{proof}
We define the necessary distributive law by replacing the labelled braids by labelled permutations in Definition \ref{dist-law-defn} and using the crossed simplicial group structure in $\mathbb{D}\otimes \mathbb{P}\otimes \mathbb{G}$. The proof then follows the same argument of Theorem \ref{bimon-thm}, using \cite[4.4,\, 4.8]{Lack} in place of \cite[5.9,\, 5.10]{DMG-PROB}.
\end{proof}

\subsection{Non-commutative sets}
The category of non-commutative sets $\fas$, originally defined in \cite[A10]{FT}, is used in \cite{Pir-PROP} to construct PROPs for monoids, comonoids and bimonoids in a symmetric monoidal category. The objects are the sets $\ul{n}$ for $n\geqslant 0$ and morphisms are maps of sets with total-ordering data on the preimages. In this subsection we extend the definition of non-commutative sets to include labelling data from a group $G$ on the pre-existing preimage data in order to encode the added structure of a group action.

\begin{defn}
Let $G$ be a group. We define an \emph{ordered $G$-labelled set} to be a finite, ordered subset of the sets $\ul{n}$, whose elements are adorned with superscript labels from $G$. We define an action of $G$ on such sets by
\[g\ast \llb j_1^{\alpha_{j_1}}<\cdots <j_r^{\alpha_{j_r}}\rrb =\llb j_1^{g\alpha_{j_1}}<\cdots <j_r^{g\alpha_{j_r}}\rrb.\]
In other words, we multiply the labels by $g$ on the left.

If we remove the ordering, we call these objects \emph{$G$-labelled sets}. Note that we have a similar action of $G$ for these objects.  
\end{defn}

\begin{defn}
Let $G$ be a group. The \emph{PROP of $G$-labelled non-commutative sets}, denoted $G\fas$, has as objects the sets $\ul{n}$ for $n\geqslant 0$. An element $f\in \mathrm{Hom}_{G\fas}\left(\ul{n} , \ul{m}\right)$ is a map of finite sets such that for each $i \in \ul{m}$, $f^{-1}(i)$ is an ordered $G$-labelled set. In particular, $\mathrm{Hom}_{G\fas}\left(\ul{0} , \ul{m}\right)$ is the singleton set containing the map $\emptyset \rightarrow \ul{m}$ and $\mathrm{Hom}_{G\fas}\left(\ul{n} , \ul{0}\right)$ is the empty set. For $f_1 \in \mathrm{Hom}_{G\fas}\left(\ul{n} , \ul{m}\right)$ and $f_2 \in \mathrm{Hom}_{G\fas}\left(\ul{m} , \ul{l}\right)$, the composite $f_2\bullet f_1$ in $G\fas$ is the composite of the underlying set maps $f_2\circ f_1$ with 
\[\left(f_2\bullet f_1\right)^{-1}(i) = \coprod_{j^{\alpha_j}\in f_2^{-1}(i)} \left(\alpha_j \ast f_1^{-1}(j)\right).\]
The symmetry isomorphisms are given by the block permutations.
\end{defn}

\begin{defn}
Let $G\mathcal{F}$ be the category whose objects are the sets $\ul{n}$. A morphism in $G\mathcal{F}$ is a map of sets such that the preimage of each singleton in the codomain is a $G$-labelled set. Composition is given by composition of set maps and multiplication of labels.
\end{defn}

\begin{prop}
\label{PROP-iso-prop}
There are isomorphisms of PROPs 
\begin{enumerate}
\item $G\fas \cong \mathbb{D}\otimes \mathbb{P}\otimes \mathbb{G}$ and
\item $G\mathcal{F}\cong \mathbb{F}\otimes \mathbb{G}$.
\end{enumerate}
\end{prop}
\begin{proof}
The proof of the first isomorphism follows \cite[Proposition 3.11]{DMG-IFAS}, with $C_2$ replaced by $G$, the PROP of hyperoctahedral groups $\mathbb{H}$ replaced with the PROP $\mathbb{P}\otimes \mathbb{G}$ and using the crossed simplicial group of \cite[Theorem 3.10]{FL} in place of \cite[3.1]{FL}. The second isomorphism is similar.
\end{proof}

\begin{cor}
\label{symm-bimon-cor}
Let $\mathbf{S}$ be a symmetric monoidal category. Let $G$ be a group. There are equivalences of categories
\begin{enumerate}
\item $\mathbf{Alg}\left(G\fas , \mathbf{S}\right) \simeq \mathbf{GMon}\left(\mathbf{S}\right)$,
\item $\mathbf{Alg}\left(G\fas^{op} , \mathbf{S}\right) \simeq \mathbf{GComon}\left(\mathbf{S}\right)$ and
\item $\mathbf{Alg}\left(G\fas \otimes_{\mathbb{P}\otimes \mathbb{G}} G\fas^{op} , \mathbf{S}\right) \simeq \mathbf{GBimon}\left(\mathbf{S}\right)$.
\end{enumerate}
\end{cor}
\begin{proof}
This follows from Proposition \ref{comp-prop}, Proposition \ref{PROP-iso-prop} and Theorem \ref{comm-bimon-thm}.
\end{proof}

We will give an alternative proof of the final part of this corollary, using the theory of double categories. This machinery will also allow us to give nice descriptions of the PROPs for $G$-commutative and $G$-cocommutative bimonoids in a symmetric monoidal category. 

\begin{cor}
Let $\mathbf{S}$ be a symmetric monoidal category. Let $G$ be a group.
\begin{enumerate}
\item The category $\mathbf{Alg}\left(G\mathcal{F}, \mathbf{S}\right)$ is equivalent to the category of $G$-commutative monoids in $\mathbf{S}$;
\item the category $\mathbf{Alg}\left(G\mathcal{F}^{op}, \mathbf{S}\right)$ is equivalent to the category of $G$-cocommutative \newline comonoids in $\mathbf{S}$.
\end{enumerate}
\end{cor}
\begin{proof}
This follows from Proposition \ref{PROP-iso-prop} and Proposition \ref{commutative-prop}.
\end{proof}

\begin{rem}
Note that when $G=C_2$, the cyclic group of order two, the PROP $G\fas$ is not isomorphic to the PROP $\ifas$ of \cite[3.5]{DMG-IFAS}. In the latter category, the action of $C_2$ also reverses the total-ordering data, encoding an order-reversing involution. However, the PROP $G\mathcal{F}$ is isomorphic to the PROP $\mathcal{I}\mathcal{F}$ of \cite[3.13]{DMG-IFAS}, since these PROPs do not have any total-ordering data.
\end{rem}

\subsection{Double categories}
In this subsection we follow the constructions used in \cite[Section 4]{Pir-PROP} and \cite[Section 4]{DMG-IFAS} to construct a double category from the PROP $G\fas$, which satisfies the star condition. Recall from \cite[2.3]{FL} that a double category satisfies the star condition if a horizontal morphism and a vertical morphism with the same codomain determine a unique bimorphism.

\begin{defn}
The double category $G\fas_2$ has as objects the sets $\ul{n}$ for $n\geqslant 0$. The sets of horizontal and vertical morphisms in $G\fas_2$ are both equal to the set of all morphisms in $G\fas$. A bimorphism in $G\fas_2$ is a square
\begin{center}
\begin{tikzcd}
\ul{n}\arrow[d, "\phi_1 ", swap]\arrow[r, "f_1"] & \ul{p}\arrow[d, "\phi "]\\
\ul{m}\arrow[r, "f", swap] & \ul{q}
\end{tikzcd}
\end{center}
of morphisms in $G\fas$ such that
\begin{itemize}
\item the underlying diagram of finite sets is a pullback square,
\item for all $x\in \ul{m}$ the map $\phi_1^{-1}(x)\rightarrow \phi^{-1}(f(x))$ induced by $f_1$ is an isomorphism of ordered $G$-labelled sets and
\item for all $y\in \ul{p}$ the map $f_1^{-1}(y)\rightarrow f^{-1}(\phi(y))$ induced by $\phi_1$ is an isomorphism of ordered $G$-labelled sets.
\end{itemize}
\end{defn}

\begin{rem}
Note that the squares in the definition of a bimorphisms need not necessarily be commutative.
\end{rem}

\begin{defn}
We define the double category $G\mathcal{F}_2$ similarly to $G\fas_2$; the objects are those of $G\mathcal{F}$, the sets of horizontal and vertical morphisms are the set of morphisms in $G\mathcal{F}$ and the bimorphisms are defined similarly to the bimorphisms of $G\fas_2$, with ordered $G$-labelled sets replaced by $G$-labelled sets. 
\end{defn}

We define two further double categories, which combine the PROPs $G\fas$ and $G\mathcal{F}$.

\begin{defn}
The double category $\mathcal{V}$ has as objects the objects of $G\fas$. The set of vertical morphisms is the set of morphisms in $G\fas$. The set of horizontal morphisms is the set of morphisms in $G\mathcal{F}$. The bimorphisms are defined similarly to those of $G\fas_2$ except that the horizontal morphisms are now in $G\mathcal{F}$.

The double category $\mathcal{H}$ is defined similarly; the set of horizontal morphisms is the set of morphisms in $G\fas$, the set of vertical morphisms is the set of morphisms in $G\mathcal{F}$ and the bimorphisms are defined similarly to those of $G\fas_2$ except that the vertical morphisms are in $G\mathcal{F}$.
\end{defn}

\begin{rem}
Given a horizontal morphism $f\colon \ul{m}\rightarrow \ul{q}$ and a vertical morphism $\phi \colon \ul{p}\rightarrow \ul{q}$ in $G\fas_2$ we determine a unique bimorphism by first taking the pullback of the underlying maps of sets. As for the double category of Example 2 in \cite[Section 4]{Pir-PROP} and the double category $\ifas_2$ of \cite[4.1]{DMG-IFAS}, the maps of the pullback have a unique lift to the category $G\fas$ where the preimage data is induced from $f$ and $\phi$ using the conditions on bimorphisms. Hence the double category $G\fas_2$ satisfies the star condition. A similar justification show that the double categories $G\mathcal{F}_2$, $\mathcal{V}$ and $\mathcal{H}$ also satisfy the star condition.
\end{rem}

\subsection{An alternative description and commutativity}
We can use the double category $G\fas_2$ and Lack's machinery of composing PROPs \cite{Lack} to give an alternative construction of the PROP for $G$-bimonoids in a symmetric monoidal category.

\begin{prop}
\label{comp-prop-prop}
There exist composite PROPs
\begin{multicols}{2}
\begin{enumerate}
\item $G\fas \otimes_{\mathbb{P}\otimes \mathbb{G}} G\fas^{op}$;
\item $G\mathcal{F} \otimes_{\mathbb{P}\otimes \mathbb{G}} G\fas^{op}$;
\item $G\fas \otimes_{\mathbb{P}\otimes \mathbb{G}} G\mathcal{F}^{op}$;
\item $G\mathcal{F} \otimes_{\mathbb{P}\otimes \mathbb{G}} G\mathcal{F}^{op}$.
\end{enumerate}
\end{multicols}
\end{prop}
\begin{proof}
We give the proof of the first case. The remaining three are similar. Following the method of \cite[Proposition 5.1]{DMG-IFAS} we use the star condition for the double category $G\fas_2$ to define a distributive law of PROs 
\[G\fas^{op}\otimes G\fas\rightarrow G\fas \otimes G\fas^{op}.\]
Both $G\fas$ and $G\fas^{op}$ are PROPs and so $G\fas \otimes G\fas^{op}$ has a PROP structure.

A morphism in $G\fas \otimes G\fas^{op}$ from $\ul{n}$ to $\ul{m}$ can be written as a span
\begin{center}
\begin{tikzcd}
\ul{n} & \ul{p}\arrow[l, "\phi ", swap]\arrow[r, "f"] & \ul{m}
\end{tikzcd}
\end{center}
in $G\fas$. Two spans
\begin{center}
\begin{tikzcd}
\ul{n} & \ul{p}\arrow[l, "\phi ", swap]\arrow[r, "f"] & \ul{m} & \text{and} & \ul{n} & \ul{p}\arrow[l, "\phi_1 ", swap]\arrow[r, "f_1"] & \ul{m}
\end{tikzcd}
\end{center}
are said to be \emph{equivalent} if there exists a bijection $h\colon \ul{p}\rightarrow \ul{p}$ in $G\fas$ such that $\phi_1\circ h=\phi$ and $f_1\circ h=f$. Note that the bijections in the category $G\fas$ correspond precisely to the morphisms in the PROP $\mathbb{P}\otimes \mathbb{G}$.

Let $G\fas\otimes_{\mathbb{P}\otimes \mathbb{G}}G\fas^{op}$ be the category obtained from $G\fas \otimes G\fas^{op}$ by identifying equivalent spans. It follows from \cite[Theorem 4.6]{Lack} that $G\fas\otimes_{\mathbb{P}\otimes \mathbb{G}}G\fas^{op}$ is a composite PROP defined via a distributive law induced from the one defined for PROs.
\end{proof}

\begin{defn}
Let $Q_{\mathcal{V}}=G\mathcal{F}\otimes_{\mathbb{P}\otimes\mathbb{G}}G\fas^{op}$. Let $Q_{\mathcal{H}}=G\mathcal{F}^{op}\otimes_{\mathbb{P}\otimes \mathbb{G}}G\fas$. Let $Q_{G\mathcal{F}}= G\mathcal{F}\otimes_{\mathbb{P}\otimes \mathbb{G}}G\mathcal{F}^{op}$.
\end{defn}

The following theorem gives an alternative proof of the third part of Corollary \ref{symm-bimon-cor}.

\begin{thm}
\label{symm-bimon-thm}
Let $\mathbf{S}$ be a symmetric monoidal category. There is an equivalence of categories
\[\mathbf{Alg}\left(G\fas\otimes_{\mathbb{P}\otimes \mathbb{G}}G\fas^{op} , \mathbf{S}\right) \simeq \mathbf{GBimon}\left(\mathbf{S}\right).\]
\end{thm}
\begin{proof}
We argue analogously to \cite[5.3]{DMG-IFAS}. By \cite[Proposition 4.7]{Lack}, an algebra for $G\fas\otimes_{\mathbb{P}\otimes \mathbb{G}}G\fas^{op}$ in $\mathbf{S}$ consists of an object $M$ with a $G\fas$-algebra structure and a $G\fas^{op}$-algebra structure subject to compatibility conditions. A $G\fas$-algebra structure is a $G$-monoid structure and a $G\fas^{op}$-algebra structure is a $G$-comonoid structure. Arguing analogously to \cite[5.3]{Lack} and \cite[5.3]{DMG-IFAS}, the compatibility conditions arising from the distributive law are precisely the conditions requiring $M$ to be a $G$-bimonoid. Finally we observe that a morphism in $\mathbf{S}$ is a morphism of $G$-bimonoids if and only if it preserves the $G\fas$-algebra structure and the $G\fas^{op}$-algebra structure. By \cite[4.8]{Lack} this is true if and only if it preserves the $G\fas\otimes_{\mathbb{P}\otimes \mathbb{G}}G\fas^{op}$-algebra structure.
\end{proof}

\begin{thm}
Let $\mathbf{S}$ be a symmetric monoidal category. Let $G$ be a group.
\begin{enumerate}
\item The category $\mathbf{Alg}\left(Q_{\mathcal{V}} , \mathbf{S}\right)$ is equivalent to the category of $G$-commutative bimonoids in $\mathbf{S}$.
\item The category $\mathbf{Alg}\left(Q_{\mathcal{H}} , \mathbf{S}\right)$ is equivalent to the category of $G$-cocommutative bimonoids in $\mathbf{S}$.
\item The category $\mathbf{Alg}\left(Q_{G\mathcal{F}} , \mathbf{S}\right)$ is equivalent to the category of $G$-commutative and cocommutative bimonoids in $\mathbf{S}$.
\end{enumerate}
\end{thm}
\begin{proof}
These equivalences are proved similarly to Theorem \ref{symm-bimon-thm}.
\end{proof}

\section{Group actions and involutions in symmetric monoidal categories}
\label{hyp-bimon-sec}

In symmetric monoidal categories there are examples of monoids that come equipped with an involution and an action of a group $G$ which commute with one another.

For example, if we take a commutative ring $k$ and a group $G$, the group algebra, $k[G]$, comes equipped with an involution (given by the $k$-linear extension of the group homomorphism $g\mapsto g^{-1}$) and an action of $G$ given by the $k$-linear extension of the conjugation action on $G$. The following result shows that we can extend the techniques of this paper to encode this structure using PROPs.

Let $\mathbf{S}$ be a symmetric monoidal category. Recall from \cite[1.1]{DMG-IFAS} that we denote the categories of involutive monoids, involutive comonoids and involutive bimonoids in $\mathbf{S}$ by $\mathbf{IMon}\left(\mathbf{S}\right)$, $\mathbf{IComon}\left(\mathbf{S}\right)$ and $\mathbf{IBimon}\left(\mathbf{S}\right)$ respectively.

\begin{thm}
\label{hyp-bimon-thm}
Let $\mathbf{S}$ be a symmetric monoidal category. Let $G$ be a group. There are equivalences of categories
\begin{enumerate}
\item $\mathbf{Alg}\left( \mathbb{D}\otimes \mathbb{H}\otimes \mathbb{G} , \mathbf{S}\right) \simeq \mathrm{Fun}\left(G, \mathbf{IMon}\left(\mathbf{S}\right)\right)$;
\item $\mathbf{Alg}\left( \left(\mathbb{D}\otimes \mathbb{H}\otimes \mathbb{G}\right)^{op} , \mathbf{S}\right) \simeq \mathrm{Fun}\left(G, \mathbf{IComon}\left(\mathbf{S}\right)\right)$ and
\item $\mathbf{Alg}\left(\mathbb{D}\otimes (\mathbb{H}\otimes \mathbb{G})\otimes \mathbb{D}^{op} , \mathbf{S}\right) \simeq \mathrm{Fun}\left(G, \mathbf{IBimon}\left(\mathbf{S}\right)\right)$.
\end{enumerate}  
\end{thm}
\begin{proof}
We define the necessary distributive law by replacing the labelled braids by labelled hyperoctahedral group elements in Definition \ref{dist-law-defn} and using the crossed simplicial group structure in $\mathbb{D}\otimes \mathbb{H}\otimes \mathbb{G}$. The proof then follows the same argument of Theorem \ref{bimon-thm}, using \cite[4.4,\, 4.8]{Lack} in place of \cite[5.9,\, 5.10]{DMG-PROB}. 
\end{proof}

\bibliographystyle{alpha}
\bibliography{gmon-refs}

\end{document}